\documentclass[10pt,notitlepage,twoside,a4paper]{amsart}

\usepackage{amsfonts}
\usepackage{amsmath, amsthm, amssymb, amscd, enumerate}
\usepackage{latexsym}
\usepackage{graphicx}

\usepackage[mathcal]{eucal}
\usepackage{comment}

\newtheorem{theorem}{\rm\bf Theorem}[section]
\newtheorem{proposition}[theorem]{\rm\bf Proposition}
\newtheorem{lemma}[theorem]{\rm\bf Lemma}

\theoremstyle{definition}

\theoremstyle{remark}
\newtheorem{remark}[theorem]{\rm\bf Remark}

\title[]{On Thurston's stretch lines in Teichm\"uller space.}

\date{\today}
\author{Guillaume Th\'eret}
\address{Guillaume Th\'eret, Lyc\'ee Ni\'epce, 71100 Chalon-sur-Sa\^one, France}
\email{guillaume.theret71@orange.fr}

\begin{document}

\begin{abstract}  
The Teichm\"uller space $\mathcal{T}(\Sigma)$ of a surface $\Sigma$ is equipped with Thurston's asymmetric metric. 
Stretch lines are oriented geodesics for this metric on $\mathcal{T}(\Sigma)$.
We give the asymptotic behavior of the lengths of the measured geodesic laminations as one follows a stretch line in the positive direction. 
\end{abstract}
\maketitle

This is a substentially revised version of an old paper from my thesis that has never been published but that I several times referred to.

\section{Introduction}
Consider an oriented connected surface $\Sigma$ without boundary, of finite topological type and of negative Euler-Poincar\'e characteristic.
Let $\mathcal{T}(\Sigma)$ be the Teichm\"uller space of $\Sigma$, that is, the space of isotopy classes of complete and finite-area hyperbolic metrics on $\Sigma$.
Endow the space $\mathcal{T}(\Sigma)$ with Thurston's \emph{asymmetric} metric $L$, defined for instance by
$$
\forall\ x,y\in\mathcal{T}(\Sigma),\quad L(x,y):=\log\inf_{\phi\sim\mathrm{id}_{\Sigma}}L_{x,y}(\phi),
$$
where $L_{x,y}(\phi) : = \sup_{p\neq q}\frac{d_y(\phi(p),\phi(q))}{d_x(p,q)}$ is the Lipschitz constant from the metric $x$ to the metric $y$ of the homeomorphism $\phi$ isotopic to the identity.

\textbf{Stretch lines}, \textbf{stretch rays} are \emph{oriented} geodesics of $\mathcal{T}(\Sigma)$ for the metric $L$.
A stretch line is completely characterized by the mean of two geodesic laminations on $\Sigma$:
a complete geodesic lamination $\mu$, called the \textbf{support} of the stretch line, and the projective class $\lambda$ of a measured geodesic lamination totally transverse to $\mu$, called the \textbf{direction} of the stretch line.
In what follows, we will parameterize a stretch line (respectively, stretch rays) isometrically (i.e., by arc-length) by the real line $\mathbb{R}$ (respectively, by $\mathbb{R}_{\geq0}$) so that orientations match. 
 
Our goal here is to give the asymptotic behavior of the length $\ell_{h_{t}}(\alpha)$ of a measured geodesic lamination $\alpha$ as one indefinitely follows a stretch line $t\mapsto h_t$ in the \emph{positive} direction.
This behavior is described in terms of the intersection pattern of $\alpha$ with the support $\mu$ and the direction $\lambda$ of the stretch line.

\begin{theorem}
\label{theorem:behavior}
Let $t\mapsto h_{t}$, $t\in\mathbb{R}_{\geq0}$, be a stretch ray of Teichm\"uller space $\mathcal{T}(\Sigma)$ whose support is denoted by $\mu$ and its direction by $\lambda$.
Let $\alpha$ be a measured geodesic lamination of $\Sigma$.
\begin{itemize}
\item If $\alpha$ has non-empty transverse intersection with the direction $\lambda$, 

then $\displaystyle{\lim_{t\to+\infty}\ell_{h_{t}}(\alpha)=+\infty}$.
\item If $\alpha$ is contained in the direction $\lambda$, 

then $\displaystyle{\lim_{t\to+\infty}\ell_{h_{t}}(\alpha)=0}$.
\item If $\alpha$ has empty intersection with the direction $\lambda$, 

then the function $t\mapsto\ell_{h_{t}}(\alpha)$ is bounded in $]0\,;+\infty[$.
\end{itemize}
\end{theorem}

\begin{remark}
One can find in \cite{Theret07} a similar description of the asymptotic behavior of the length of a measured geodesic lamination as one follows a stretch line in the \emph{negative} direction.
The statement of the theorem in that case is obtained by swapping the r\^oles of the support and of the direction.
However, the arguments used to establish these behaviors are rather different since a stretch line traversed in opposite direction is not necessarily a geodesic. 
\end{remark}

\begin{remark}
In \cite{Theret14}, I show that the length function of a measured geodesic lamination is convex along a stretch line when the line is parameterized by \emph{the logarithm} of the arc length.
This in particular implies, thanks to the above Theorem \ref{theorem:behavior} and to the similar theorem for the stretch lines followed in negative direction, that the length of a measured geodesic lamination which is contained in the support is strictly decreasing and that the length of a measured geodesic lamination which is disjoint from both the support and the direction is constant. 
\end{remark}

\section{On stretch lines}

We briefly recall in this section the definition of a stretch line.
The reader is warmly encouraged to read Thurston's original paper \cite{Thurston86} for further details.

In the whole paper, the hyperbolic metrics we shall consider on the surface $\Sigma$ will always tacitely be complete and of finite area. 
If $\mu$ denotes a geodesic lamination of $\Sigma$ and if $\Sigma$ is equipped with a hyperbolic metric, we shall also denote the geodesic representative of $\mu$ in its isotopy class, with respect to the metric, by the same letter $\mu$.
The \emph{measured} geodesic laminations will always be compact.

Fix a \emph{complete} geodesic lamination $\mu$ of the surface $\Sigma$, that is, a geodesic lamination whose complementary regions are interiors of ideal triangles (for any hyperbolic metric on $\Sigma$).
Fix also a point $h_{0}$ of $\mathcal{T}(\Sigma)$.
One associates to the lamination $\mu$ and the point $h_{0}$ of $\mathcal{T}(\Sigma)$ a measured lamination $\lambda_{\mu}(h_{0})$ of $\Sigma$ as follows.

Consider a representative $m_{0}$ of $h_{0}$, that is, a hyperbolic metric on $\Sigma$. 
Partially foliate each ideal triangle of $\Sigma\setminus\mu$ with arcs of horocycles centered at the vertices of the ideal triangle, in such a way that a region bounded by three arcs of length one meeting tangentially remains unfilled (see Figure \ref{Figure1}).
Each of the three \emph{closed} foliated parts lying over the three sides of the non-foliated region is called a \textbf{spike}. 
The three vertices of the non-foliated region are called the {\bf distinguished points}.  

\begin{figure}[!h]
\centering

\includegraphics[width = 5cm, height = 4cm]{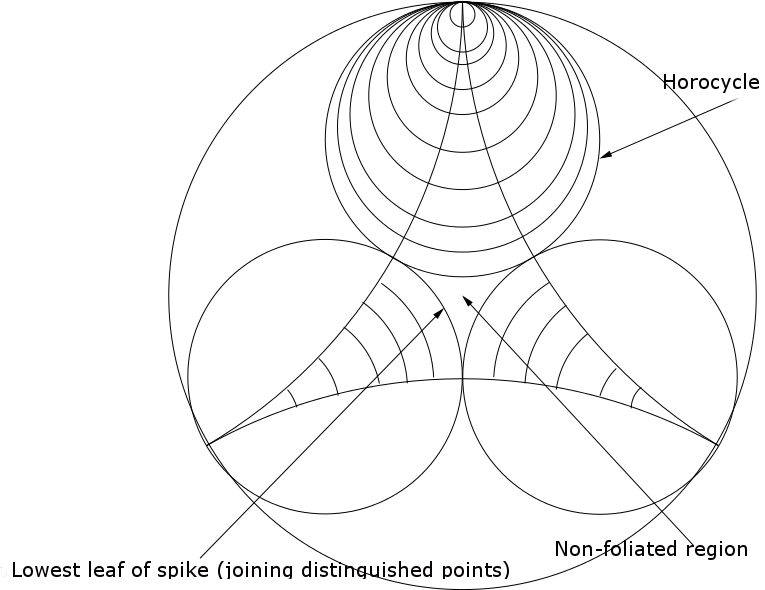}
\caption{The horocyclic foliation within each ideal triangle.}
\label{Figure1}
\end{figure}

This partial foliation defined in the complement of $\mu$ in $\Sigma$ extends continuously over $\mu$.
Equip this so-obtained partial foliation with the transverse measure that coincides with the element of length on the leaves of $\mu$. 
This partial measured foliation of $\Sigma$ is denoted by $F_{\mu}(m_{0})$ and is called the \textbf{horocyclic foliation} associated to $\mu$ and $m_{0}$.
As the hyperbolic metric $m_{0}$ varies in its isotopy class $h_{0}$, the horocyclic foliation $F_{\mu}(m_{0})$ varies in its measure class $F_{\mu}(h_{0})\in\mathcal{MF}(\Sigma)$. 
We thus have a map
$$
F_{\mu}\ :\ \mathcal{T}(\Sigma)\to\mathcal{MF}(\Sigma),\quad h_{0}\mapsto F_{\mu}(h_{0}).
$$
The measured geodesic lamination naturally associated to $F_{\mu}(h_{0})$ is denoted by $\lambda_{\mu}(h_{0})$ and is called the \textbf{horocyclic lamination} associated to $\mu$ and $h_{0}$.

Thurston proved that the above map $F_{\mu}$ is a homeomorphism onto its image, the latter being the space of the classes of the measured foliations transverse to $\mu$ and standard near the cusps of $\Sigma$.
The inverse image by $F_{\mu}$ of the ray $t\mapsto e^{t}F_{\mu}(h_{0})$, $t\in\mathbb{R}$, is the \textbf{stretch line} supported by $\mu$ and passing through $h_{0}\in\mathcal{T}(\Sigma)$.
Let $t\mapsto h_{t}$, $t\in\mathbb{R}$, denote this stretch line. 
It is endowed with the orientation induced by $\mathbb{R}$.
By definition, $h_{t}=F_{\mu}^{-1}(e^{t}F_{\mu}(h_{0}))$. 
Put in other words, the horocyclic lamination associated to $h_{t}$ is $e^{t}\lambda_{\mu}(h_{0})$, that is, the horocyclic lamination associated to $h_{0}$ but with the transverse measure multiplied by $e^{t}$. 
The \textbf{direction} of the stretch line is the projective class of the horocyclic lamination $\lambda_{\mu}(h_{0})$.
A stretch line is thus characterized by its support and its direction. 

The stretch lines form a particular class of geodesics in Teichm\"uller space equipped with Thurston's asymmetric metric $L$.

\begin{remark}
In Thurston's paper \cite{Thurston86}, a subclass of stretch lines is more particularly studied,
namely, the stretch lines whose supports are \emph{chain-recurrent}.
This topological condition prevents the spiralling of the leaves of the support around a closed leaf in the same direction.
We shall also assume that all supports are chain-recurrent.
\end{remark}

\section{Asymptotic behaviors of the lengths of the support and of the direction}

\subsection{Lengths of measured geodesic laminations, of measured partial foliations}

Consider a hyperbolic metric on the surface $\Sigma$ and a measured geodesic lamination $\alpha$ of $\Sigma$.
One way to define the length of $\alpha$ is to consider a particular fat traintrack approximation $\Theta$ of $\alpha$.
This particular traintrack is a regular neighborhood of $\alpha$, small enough to be of stable topological type, foliated by arcs transverse to the boundary called the \textbf{traverses} of $\Theta$.
It is obtained as follows.
Choose a small regular neighborhood of the boundary of each component of $\Sigma\setminus\alpha$.
The boundary of such a component is made up of finitely many (possibly none) spikes (asymptotic half-lines) and of finitely many (possibly none) simple closed geodesics.
The neighborhoods are foliated by arcs of horocycles in the vicinity of the spikes (just like the spikes of the ideal triangles of $\Sigma\setminus\mu$ for the horocyclic foliation).
These arcs are perpendicular to the boundary of the neighborhoods.
It is however impossible to foliate the whole neighborhoods with such arcs because two foliations about spikes do not merge correctly (hence the presence of non-foliated regions in the construction of the horocyclic foliation).
This gives rise to small ``compromise zones" inbetween, made up of arcs we still choose to be perpendicular to the boundary.   
When the neighborhoods have no spikes, they are collars about simple closed geodesics of $\alpha$. 
They are then foliated with geodesic arcs perpendicular to the boundary.
Piecing these neighborhoods together in $\Sigma$ yields a regular neighborhood $\Theta$ of $\alpha$ equipped with a foliation whose leaves are perpendicular to the leaves of $\alpha$ and to the boundary of $\Theta$.
These leaves are called the {\bf traverses} of $\Theta$.
The boundary of $\Theta$ is made up of smooth arcs and finitely many singular points forming cusps.
The traverses passing through the singular points are called the \textbf{singular traverses}.
The singular traverses cut the traintrack into finitely many rectangles $R_{1},\cdots,R_{n}$, 
each foliated by traverses and crossed by pieces of leaves of $\alpha$.
The nice property of this particular traintrack approximation of $\alpha$ is that,
in each rectangle $R_{i}$, the lengths of all the pieces of leaves of $\alpha$ are the same.
Denote this length by $\ell_{i}$.
Moreover, all the traverses of $R_{i}$ have the same transverse measure $m_{i}$.

By definition, the length of $\alpha\cap R_{i}$ is the sum of the lengths of the geodesic segments of $\alpha\cap R_{i}$ against the transverse measure $d\alpha$ of $\alpha$.
In formula, 
$$
\ell_{h_{0}}(\alpha)=\sum_{i=1}^{n}m_{i}\cdot\ell_{i},
$$
where $h_{0}\in\mathcal{T}(\Sigma)$ is the isotopy class of the underlying metric, 
since the length does not depend upon the representative of $h_{0}$.  
It can also be checked that this definition does not depend upon the choice of the traintrack approximation $\Theta$ of $\alpha$.

We therefore have a function $\ell(\alpha)$ defined over the Teichm\"uller space of $\Sigma$.
The length of a measured geodesic lamination $\alpha$ is the minimum of the lengths of the partial measured foliations $F$ isotopic to $\alpha$ with the transverse measures preserved. 
Actually, this property is true even if the partial ``foliation" $F$ has non-simple leaves or has leaves which intersect one another, as soon as each leaf of $F$ is homotopic to a leaf of $\alpha$ 
(see \cite{Theret14} for detail and proof).
 
\begin{proposition}[\textbf{Minimal lengths realized by geodesic laminations} \cite{Pap91}, \cite{Theret14}]
\label{proposition:min_length}
Let $\alpha$ be a measured geodesic lamination of $\Sigma$.
For any measured partial foliation $F$ of $\Sigma$ isotopic to $\alpha$ with the transverse measure preserved, 
one has
$$
\ell_{h}(\alpha)\leq\ell_{h}(F),\ \forall h\in\mathcal{T}(\Sigma),
$$
with equality if and only if $F$ is equal to $\alpha$. 
\end{proposition}
\vskip .5cm

It turns out that the length of the horocyclic foliation $F_\mu(h)$ with respect to the metric $h$ is easily computed.\\

\begin{lemma}[\textbf{Length of horocyclic foliation} \cite{Pap91}]
\label{lemma:length_horocyclic}
Let $h$ be a point of Teichm\"uller space $\mathcal{T}(\Sigma)$ and let $\mu$ be a complete geodesic lamination of $\Sigma$.
The length of the horocyclic foliation $F_{\mu}(h)$ with respect to any representative of $h$ is equal to $3\vert\chi(\Sigma)\vert$.
\end{lemma}

\begin{proof}
The transverse measure of the horocyclic foliation $F_{\mu}(h)$ coincides, 
when the surface $\Sigma$ is endowed with any representative of $h$, 
with the hyperbolic length element.
It suffices therefore to compute the \emph{area} of the subsurface foliated by $F_{\mu}(h)$.
The surfaces decomposes into $\frac{4}{3}\vert\chi(\Sigma)\vert$, each containing three isometric foliated spikes.
This area hence decomposes into the area of $4\vert\chi(\Sigma)\vert$ spikes.
The area of such a spike is easy to compute, in the upper-half plane model for instance.
By putting the ideal vertex of the spike at infinity and the boundary of the spike at the ordinate 1, the area is given by
$$
\int_{0}^{+\infty}e^{-x}dx = 1.
$$
This concludes the proof.
\end{proof}

\subsection{Asymptotic behavior of direction and support}

Consider a stretch line $t\mapsto h_{t}$, $t\in\mathbb{R}$, whose support is denoted by $\mu$ and its direction by $\lambda$.
This section is devoted to show the following lemma.

\begin{lemma}
\label{lemma:lambda_mu}
Let $t\mapsto h_{t}$, $t\in\mathbb{R}$, be a stretch line whose support is denoted by $\mu$ and its direction by $\lambda$.
Let $\alpha$ be a measured geodesic lamination of $\Sigma$.
\begin{enumerate}
\item If $\alpha$ is contained in $\mu$, then $\displaystyle{\lim_{t\to+\infty}}\ell_{h_{t}}(\alpha)=+\infty$.
\item If $\alpha$ is contained in $\lambda$, then $\displaystyle{\lim_{t\to+\infty}}\ell_{h_{t}}(\alpha)=0$.
\end{enumerate}
\end{lemma}

\begin{remark}
The complete geodesic lamination $\mu$ does not necessarily carry a transverse measure but, if $\mu$ is not a finite union of infinite leaves going at both ends towards cusps of $\Sigma$ (in which case length has no meaning in our sense), there always exists
a compact non-empty sublamination of $\mu$ carrying a transverse measure. 
\end{remark}

\begin{proof}
We first establish Statement (1). 
Let $\alpha$ be a measured geodesic lamination contained in the support $\mu$ of the stretch line $t\mapsto h_{t}$ under consideration.
Since by definition stretching multiplies arc length along the leaves of $\mu$ by $e^{t}$, we have
$$
\ell_{h_{t}}(\alpha)=e^{t}\ell_{h_{0}}(\alpha).
$$
It follows that $\displaystyle{\lim_{t\to+\infty}}\ell_{h_{t}}(\alpha)=+\infty$.\\

We now establish Statement (2).
Suppose that $\alpha$ is a measured geodesic lamination contained in the direction $\lambda$ of the stretch line.
Consider the representative $\lambda_{\mu}(h_{0})$ of $\lambda$.
We first deal with the case $\alpha=\lambda_{\mu}(h_{0})$.
By Lemma \ref{lemma:length_horocyclic}, $\ell_{h_{t}}(F_{\mu}(h_{t}))=3\vert\chi(\Sigma)\vert$.

Now thanks to Proposition \ref{proposition:min_length}, we have

\begin{eqnarray*}
 \ell_{h_{t}}(\lambda_{\mu}(h_{0}))&\leq&\ell_{h_{t}}(F_{\mu}(h_{0}))\\
 &\leq&\ell_{h_{t}}(e^{-t}F_{\mu}(h_{t}))\\
 &\leq&e^{-t}\ell_{h_{t}}(F_{\mu}(h_{t}))\\
 &\leq&3\vert\chi(\Sigma)\vert e^{-t}.
\end{eqnarray*}
Hence,
$$
\lim_{t\to+\infty}\ell_{h_{t}}(\lambda_{\mu}(h_{0}))=0.
$$

Now let $\alpha$ be an arbitrary measured geodesic lamination contained in $\lambda_{\mu}(h_{0})$.
Let $h$ be an arbitrary point of $\mathcal{T}(\Sigma)$.
Choose a traintrack approximation $\Theta$ of $\lambda_{\mu}(h_{0})$ as above described, for some hyperbolic metric in $h$.
Let $R_{1},\cdots,R_{n}$ be the rectangles of $\Theta$ cut by the singular traverses and let $\ell_{i}$ and $m_{i}$ respectively be the lengths of the arcs of $\lambda_{\mu}(h_{0})\cap R_{i}$ with respect to $h$ and the transverse measure of the traverses of $R_{i}$ with respect to $\lambda_{\mu}(h_{0})$.
Since $\alpha$ is contained (as a set) in $\lambda_{\mu}(h_{0})$, the traintrack $\Theta$ is a traintrack carrying $\alpha$ and can be used to compute $\ell_{h_{0}}(\alpha)$.
Each rectangle $R_{i}$ is crossed by leaves of $\alpha$ (possibly none) and $\alpha\cap R_{i}\subset\lambda_{\mu}(h_{0})\cap R_{i}$, so these arcs have length $\ell_{i}$.
Let $m'_{i}$ be the transverse measure of the traverses of $R_{i}$ with respect to $\alpha$.
Set $c_{i}:=m'_{i}/m_{i}$ and $C:=\max\{c_{i}\!:\!1\leq i\leq n\}$.
We have
$$
\ell_{h}(\alpha)=\sum_{i=1}^{n}m'_{i}\cdot\ell_{i}=\sum_{i=1}^{n}c_{i}m_{i}\cdot\ell_{i}\leq C\sum_{i=1}^{n}m_{i}\cdot\ell_{i}=C\ell_{h}(\lambda_{\mu}(h_{0})).
$$
Note that $C$ does not depend upon the point $h$.
Since $\displaystyle{\lim_{t\to+\infty}}\ell_{h_{t}}(\lambda_{\mu}(h_{0}))=0$, we obtain 
$$
\lim_{t\to+\infty}\ell_{h_{t}}(\alpha)=0.
$$
This concludes the proof of the lemma.
\end{proof}

\section{Horogeodesic laminations}
\subsection{Horogeodesic curves}
Assume the surface $\Sigma$ to be endowed with some hyperbolic metric whose isotopy class is denoted by $h$.
Let $\mu$ be a complete geodesic lamination of $\Sigma$.
The lamination $\mu$ and the horocyclic foliation $F_{\mu}(h)$ form a kind of grid over the surface $\Sigma$.
Given a measured geodesic lamination $\alpha$ of $\Sigma$, 
we would like to put $\alpha$ into minimal position with respect to that grid.
The idea is to replace every leaf of $\alpha$ by a curve which follows leaves of $\mu$ and of $F_{\mu}(h)$ in the most efficient way,
that is, by minimizing the intersection between the leaves of $\alpha$ and the leaves of $\mu$ and of $F_{\mu}(h)$. 
The object we obtain after this replacement will be called a horogeodesic lamination.

It is better for this purpose to work in the universal covering of $\Sigma$ which, 
since all hyperbolic metrics are complete, 
is identified with the hyperbolic plane $\mathbb{H}^{2}$.
Let $\mu$ also denote the preimage in $\mathbb{H}^{2}$ of the complete geodesic lamination $\mu$ of $\Sigma$.
Let $F$ denote the preimage of the horocyclic foliation $F_{\mu}(h)$.
A curve $t\mapsto\gamma(t)$ in $\mathbb{H}^{2}$ is \textbf{horogeodesic} (with respect to $\mu$ and $F$) if it 
is a concatenation of curves alternately contained in a leaf of $\mu$ and in a leaf of $F$.
The maximal pieces (in the sense of inclusion) of the image of $\gamma$ contained in $\mu$ will be called the \textbf{geodesic pieces} and the maximal pieces of the image of $\gamma$ contained in a leaf of $F$ will be called the \textbf{horocyclic pieces} (whence the name horogeodesic).
A curve $t\mapsto\gamma(t)$ \textbf{avoids backtracking} with respect to $\mu$ if, once it enters inside an ideal triangle of $\mathbb{H}^{2}\setminus\mu$, it eventually exits this triangle (that is, it eventually enters inside \emph{another} ideal triangle) and does not come back to this triangle ever again (that is, the intersection of the image of $\gamma$ with the interior of each ideal triangle of $\mu$ has exactly one connected component).
A curve $t\mapsto\gamma(t)$ \textbf{avoids backtracking} with respect to $F$ if it does not bound with a leaf of $F$ a disc in $\overline{\mathbb{H}^{2}}=\mathbb{H}^{2}\cup S_{\infty}$ (which necessarily does not contain any non-foliated region.).
There are two limit cases which are considered as backtracking with respect to $F$:
1) when the disc is shrunk to a single geodesic segment, forming a {\bf needle}.
2) when the disc is infinite, bordered by two parallel leaves contained in $F$, ending at the same point at infinity.
We shall say that a horogeodesic curve $t\mapsto\gamma(t)$ is \textbf{good} if it avoids backtracking with respect to both $\mu$ and $F$.
In what follows, all horogeodesic curves will be parameterized by arc length.

\begin{lemma}
\label{lemma:endpoints}
The image of a good horogeodesic curve is embedded.
Moreover, the image of such a horogeodesic curve defined over $\mathbb{R}$ converges at both ends towards points of the circle at infinity. 
\end{lemma}

\begin{proof}
Each ideal triangle of $\mathbb{H}^{2}\setminus\mu$ boards three half-spaces.
The ideal triangles of $\mathbb{H}^{2}\setminus\mu$ are nested, which means that two of them are disjoint (and, as a matter of fact, are contained in one of the three half-spaces of the other).
Furthermore, the Euclidean diameter of the half-spaces decrease to zero as they converge to infinity.
It follows that a non-backtracking horogeodesic curve has a limit point on the circle at infinity for both of its ends.
\end{proof}

Good horogeodesic curves can thus be parameterized isometrically and, from now on, they all will be tacitely parametrized this way.

A curve in $\Sigma$ is \textbf{essential} if it is not homotopic to a cusp or a point.
A curve in $\mathbb{H}^2$ is \textbf{essential} if it is the lift of an essential curve of $\Sigma$.

\begin{lemma}
\label{lemma:endpoints}
A good essential horogeodesic curve defined over $\mathbb{R}$ has two different endpoints on the circle at infinity.
\end{lemma}

\begin{proof}
Horogeodesic curves homotopic to a point in $\Sigma$ are homotopic to a point in $\mathbb{H}^2$ and conversely, so they necessarily backtrack.

A curve which is homotopic to a cusp has only one endpoint on the circle at infinity (the fixed point of the parabolic deck transformation defining the cusp).
Conversely, suppose that the good horogeodesic curve $\gamma$ has only one endpoint on the circle at infinity. 
Necessarily, in both directions, this curve gets trapped in some sequence of ideal triangles all with a common endpoint, which is the endpoint $p$ of $\gamma$.
Such a sequence of ideal triangles is the preimage of finitely many ideal triangles of $\Sigma$ under a deck transformation and the endpoint $p$ is a fixed point of this deck transformation $g$. 
If this deck transformation is hyperbolic, $p$ is one endpoint of the geodesic axis of $g$.
In this case, the ideal triangles spiral along this geodesic axis $\alpha$ (which is a closed geodesic in $\Sigma$).
Since $\mu$ is chain-recurrent, there is a neighborhood of $\alpha$ made up of ideal triangles with a common vertex at $p$ all lying on one side of $\alpha$ and ideal triangles with a common vertex at the other endpoint $q$ of $\alpha$ all lying on the other side of $\alpha$. 
Therefore, the curve $\gamma$ must lie on the same side of $\alpha$ and both its ends cross the same sequence of ideal triangles.
This implies that $\gamma$ necessarily backtracks.
We thus conclude that the point $p$ is the fixed point of a parabolic deck transformation and that $\gamma$ is homotopic to a cusp.
The proof is over.
\end{proof}

Homotopies in $\mathbb{H}^{2}$ are understood relative to the circle at infinity.
Lemma \ref{lemma:endpoints} implies that a bi-infinite essential good horogeodesic curve is homotopic to a geodesic line in $\mathbb{H}^{2}$.
In what follows, we shall stick to essential curves only.
We shall see that, conversely, a geodesic can be turned into a horogeodesic curve.

Consider now a bi-infinite horogeodesic curve $\gamma$ which does not backtrack with respect to $\mu$.
We also assume that $\gamma$ is not a leaf of $\mu$.
We want to homotope $\gamma$ to a good horogeodesic curve, that is, we want to erase all backtrackings of $\gamma$ with respect to $F$.
To do so, we introduce some moves we call \textbf{pushes}. 
These moves consist in replacing a horocyclic segment by another one with endpoints on the same leaves of $\mu$.
The horocylic segment might be reduced to a point and used to erase needles on horocyclic curves; see Figure \ref{Figure2}.

\begin{figure}[!h]
\centering
\includegraphics[width = 10cm, height = 3cm]{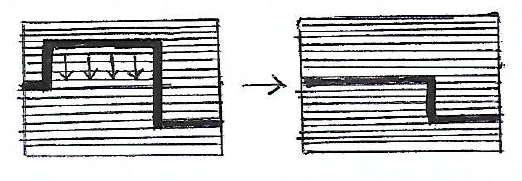}
\caption{Push.}
\label{Figure2}
\end{figure}

The horogeodesic curve $\gamma$ crosses countably many spikes $\left(S_i\right)_{i\in I}$ and crosses (possibly uncountably many) leaves of $\mu$.
Since $\gamma$ does not backtrack with respect to $\mu$, $\gamma\cap \mathring{S}_i$ consists in one connected component precisely.
Using pushes, we can homotope $\gamma$ so as to erase all needles, if any.
Let $N(\gamma)$ be the union of $\mu_\gamma \cup \bigcup_{i\in I} S_i$, where $\mu_{\gamma}$ denotes the union of the leaves of $\mu$ crossed by $\gamma$.
The set $N(\gamma)$ is a closed subset of $\mathbb{H}^2$ foliated by the restriction of $F$. 
We also denote this foliation of $N(\gamma)$ by $F$.
Furthermore, $\gamma\subset N(\gamma)$.
See Figure \ref{Figure3}

\begin{figure}[!h]
\centering
\includegraphics[width = 8cm, height = 4cm]{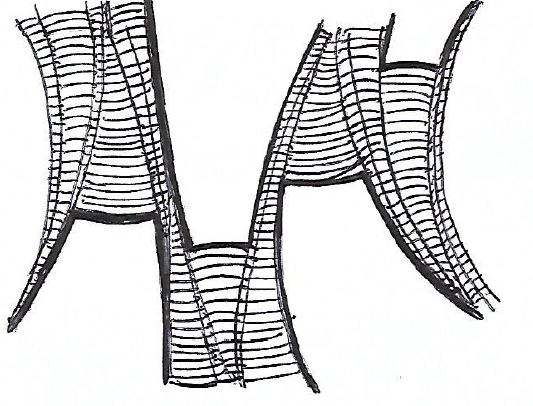}
\caption{Picture of $N(\gamma)$.}
\label{Figure3}
\end{figure}

Let us have a closer look to this set $N(\gamma)$, forgetting about $\gamma$ for a moment.

We say that spike $S_i$ is {\bf linked} to spike $S_j$ whenever there exists a leaf of $F$ passing through $S_i$ and $S_j$;
otherwise we say $S_i$ and $S_j$ are {\bf unlinked}.
See Figure \ref{Figure4}.

\begin{figure}[!h]
\centering
\includegraphics[width = 9cm, height = 4cm]{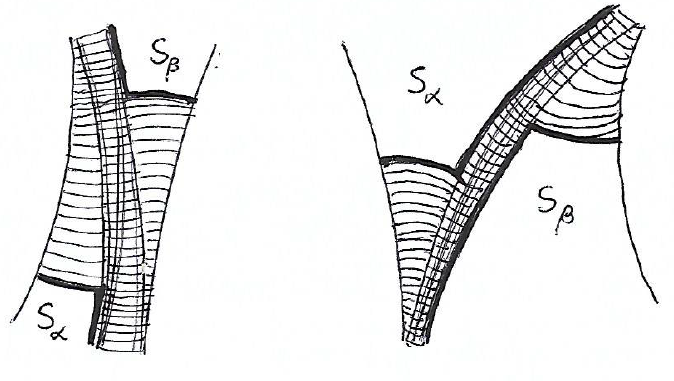}
\caption{Linked and unlinked spikes $S_i, S_j$.}
\label{Figure4}
\end{figure}

We group linked spikes together, which means that any two spikes in such a group are linked, while any spike in a group is unlinked to any spike in another group. 
Let us check the consistency of this grouping linked spikes.
First notice then that all spikes lying between $S$ and $S'$ belong to the same group, since a leaf connecting $S$ to $S'$ necessarily passes through all spikes inbetween.
Therefore spike $S$ belongs to at least one group and at most two groups:
there should be two groups if and only if there exist two spikes $S'$ and $S''$ separated by $S$ which are linked to $S$ but unlinked one to the other.
This implies that our decomposition into groups of linked spikes is well-defined.

We obtain a decomposition of $N(\gamma)$ into subsets $N_k$, $k\in K$, each of which being the union of linked spikes with the union of the leaves of $\mu$ sandwitched by these linked spikes. 
By definition, for each $k\in K$ there exists a leaf of $F$ that crosses all spikes in $N_k$ and no others.
So $N(\gamma) = \bigcup_{k\in K} N_k$.
In general, the interiors $\mathring{N}_k$ are not disjoint, since a spike might be linked to two unlinked spikes, as discussed above.
We decompose $N(\gamma)$ further so as to get a union with disjoint interiors $N(\gamma) = \bigcup_{l\in L} M_l$ (by singling out union of linked spikes that belong to two adjacent groups).
This way, $M_l\cap M_{l'}$, $l\neq l'$, is either empty or consists in one leaf of $\mu$.
Each $M_l$ contains linked spikes so, for each $M_l$ there exists a leaf $\alpha_l$ of $F$ which crosses all the spikes of $M_l$.

Using pushes we can now replace $\gamma$ by another horogeodesic curve $\gamma'$ so that $\gamma'\cap M_l = \alpha_l$, for all $l\in L$.
The geodesic segments of $\gamma'$ are therefore contained in the leaves of $\mu$ that separate the $M_l$'s.
See Figure \ref{Figure5}.

\begin{figure}[!h]
\centering
\includegraphics[width = 8cm, height = 4cm]{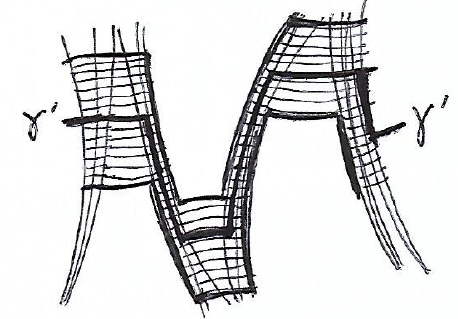}
\caption{The curve $\gamma'$. (The geodesic pieces have been drawn a bif off from the boundary of $N(\gamma)$ where they should be)}
\label{Figure5}
\end{figure}

Now we simplify the set $N(\gamma)$ into $R(\gamma)$ as follows.
For each $l\in L$, let $R_l$ be the subset of $M_l$ made up of the leaves of $F$ that cross all spikes of $M_l$ (the leaf $\alpha_l$ is one of these leaves). 
Each set $R_l$ is a foliated rectangle. 
Notice that both horocyclic sides of $R_l$ contain a side of a non-foliated region of some ideal triangle.
We set $R(\gamma) = \bigcup_{l\in L} R_l$.
This set is a closed subset of $\mathbb{H}^2$ for which all unnecessary pieces of $\mu$ and of the spikes of $N(\gamma)$ have been discarded.
The curve $\gamma'$ is contained in $R(\gamma)$.

Now the only possible way for $\gamma'$ for backtracking with respect to $F$ consists in discs contained in some $R_l$.
Using pushes in $R_l$, these backtrackings can be erased, and the thus-obtained horogeodesic curve is good.  
We thus have proved the following result.

\begin{lemma} 
\label{lemma:PushToGood}
A horogeodesic curve which avoids backtracking with respect to the lamination $\mu$ can be turned into a good horogeodesic curve by the use of pushes.
\end{lemma}

\begin{remark}
The set $N(\gamma)$ contains all horogeodesic curves which avoid backtracking with respect to $\mu$ and that are homotopic to $\gamma$. 
All these curves have the same endpoints on the circle at infinity.
They are all homotopic through pushes.
Also note that if the curve $\gamma$ is homotopic to a (possibly singular) leaf $f$ of $F$, then pushes homotope it to that leaf $f$.
\end{remark}

\subsection{Stairstep curves}

We now look closer at a special class of horogeodesic curves we call stairstep curves.
We first define shift rectangles and shift segments.

A rectangle is a closed disc with four points on the boundary, called the vertices of the rectangle, cutting the boundary into four curves called the sides of the rectangle. 
We allow "flat" rectangles as well, which are just closed segments.
A \textbf{shift rectangle} (for the horocyclic foliation $F$ and the complete geodesic lamination $\mu$) is a rectangle with two opposite sides contained in leaves of $F$, with the two other sides contained in leaves of $\mu$ and such that two opposite vertices are distinguished points and no other distinguished points belong to that rectangle. 
(A flat shift rectangle is just a segment either contained in a leaf of $F$ or of $\mu$ and joining two distinguished points.)
A \textbf{shift segment} is a curve transverse to the foliation $F$ and homotopic, through a homotopy preserving leaves of $F$, to a geodesic side of a shift rectangle.

We now define stairstep curves. 
Recall that a horogeodesic curve is a concatenation of alternately horocyclic and geodesic pieces, respectively contained in leaves of $F$ and of $\mu$.
A \textbf{stairstep curve} is a horogeodesic curve not contained in $\mu$ all of whose geodesic pieces are shift segments.

As the following lemma asserts it, stairstep curves are just, up to pushes, good horogeodesic curves.

\begin{lemma}
\label{lemma:stairstep_good}
Any bi-infinite horogeodesic curve $\gamma$ which avoids backtracking with respect to $\mu$ is homotopic to a stairstep curve $\gamma^{\star}$ through pushes. 
\end{lemma}

\begin{proof}
Let $\gamma$ be a bi-infinite horogeodesic curve which avoids backtracking with respect to $\mu$.
Lemma \ref{lemma:PushToGood} says that $\gamma$ can be pushed to a good horogeodesic curve $\gamma'$.
If $\gamma$ was homotopic to a possibly singular leaf of $F$, then $\gamma'$ is precisely that leaf of $F$ and is a stairstep curve.
So assume that $\gamma$ is not homotopic to a leaf of $F$, that is, $\gamma'$ has at least one geodesic piece.  
Consider the set $R(\gamma)=\bigcup_{l\in L} R_l$ containing $\gamma'$, as defined above.
Each horocyclic piece $h_l$ of $\gamma'$ is contained in a rectangle $R_l$ whose horocyclic sides contain distinguished points.
Moreover, each geodesic piece of $\gamma'$ is contained in some geodesic side of a rectangle and every such geodesic side meets precisely one geodesic piece of $\gamma'$.
In every rectangle $R_l$, the curve $\gamma'$, equipped with an arbitrary orientation, either 
\begin{enumerate}
\item draws a ``Z", that is, follows a part of the geodesic boundary of $R_l$, then turns (say, to the left) and enters $R_l$, then goes through it along a horocyclic piece, and finally turns (hence to the right) to follow a part of the other geodesic side of $R_l$;  
\item draws a ``U", that is, follows a part of the geodesic boundary of $R_l$, then turns (say, to the left) and enters $R_l$, then goes through it along a horocyclic piece, and finally turns (also to the left) to follow a part of the other geodesic side of $R_l$. 
\end{enumerate}
In the ``U"-case, the horocyclic part of $\gamma'$ is contained in the boundary of $R_l$ (otherwise, $\gamma'$ would be backtracking with respect to $F$).
See Figure \ref{Figure6}

\begin{figure}[!h]
\centering
\includegraphics[width = 5cm, height = 3cm]{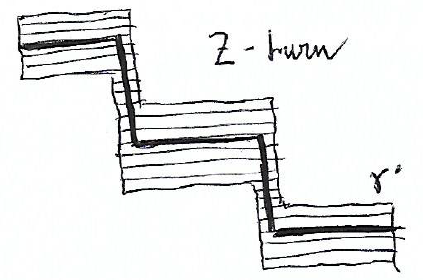}
\includegraphics[width = 5cm, height = 3cm]{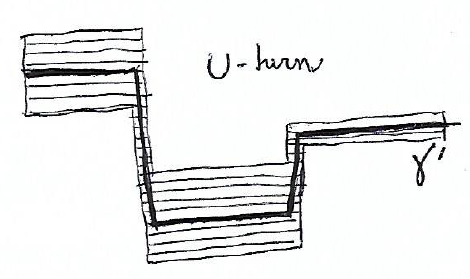}
\caption{The curve $\gamma'$ drawing "Z``- and "U``-turns.}
\label{Figure6}
\end{figure}

In any case, we use a push to modify $\gamma'$ in every rectangle $R_l$ so that the horocyclic part of the new curve $\gamma^{\star}$ is contained in the boundary of $R_l$.
This implies that exactly one geodesic side of $R_l$ is entirely followed by $\gamma^{\star}$.
The geodesic segments of $\gamma^{\star}$ are all shift segments, because all horocyclic sides of the $R_l$'s contain distinguished points.
The curve $\gamma^{\star}$ is a stairstep curve.
The proof is over.
\end{proof}

\begin{lemma}
\label{lemma:moves}
Two bi-infinite good horogeodesic curves are homotopic if and only if they are related by a sequence of pushes.
\end{lemma}

\begin{proof}
It is clear that two horogeodesic curves that are related by a sequence of pushes are homotopic.

Now consider two bi-infinite good horogeodesic curves $\alpha$ and $\beta$ which are homotopic.
Because they are homotopic and good, they define the same set $N(\alpha)=N(\beta)$ and the same decomposition in rectangles $R = \bigcup_{l\in L} R_l$.
By construction, $\alpha$ and $\beta$ are both contained in $R$.
Thanks to Lemma \ref{lemma:stairstep_good}, $\alpha$ and $\beta$ can be homotoped through pushes to the same stairstep curve $\alpha^{\star} = \beta^{\star}$.
This concludes the proof.
\end{proof}

The above lemma is useful because it shows that homotopic good horogeodesic curves have the same transverse intersection with respect to the horocyclic foliation; see next section.

\begin{remark}
Note that if the curves are lifts of curves on the surface $\Sigma$, pushes can be performed in an equivariant way with respect to deck transformations so that the deformations in $\mathbb{H}^2$ can be carried out on the surface $\Sigma$ if necessary.
\end{remark}

\subsection{Horogeodesic laminations}

A \textbf{horogeodesic lamination} of $\mathbb{H}^{2}$ is a lamination whose leaves are horogeodesic curves and which is homotopic to a geodesic lamination.
A horogeodesic lamination of the surface $\Sigma$ is the projection on $\Sigma$ of a horogeodesic lamination of $\mathbb{H}^{2}$ which is invariant under the deck transformations.
Note that the leaves of a horogeodesic lamination on $\Sigma$ might not be simple.\\

Let us now return back to the surface $\Sigma$ equipped with some hyperbolic metric $h$.
Equip this hyperbolic surface with a complete geodesic lamination $\mu$ and its horocyclic foliation $F=F_{\mu}(h)$. 
Consider a horogeodesic lamination $\alpha^{\star}$, homotopic to a given geodesic lamination $\alpha$.
(We shall see below a way of producing such a $\alpha^{\star}$ from a given $\alpha$.)
Assume that $\alpha$ has a transverse measure.
This transverse measure naturally defines a transverse measure on $\alpha^{\star}$ (one can use the circle at infinity in the universal covering to define the transverse measure on local cross sections; compare \cite{Bonahon}).
We shall then say that the horocyclic lamination $\alpha^{\star}$ equipped with the transverse measure induced by $\alpha$ belongs to the measure class of $\alpha$.
Thanks to this transverse measure, one can compute the length $\ell_{h}(\alpha^{\star})$ as follows.

Fix a rectangular covering $\mathcal{R}=\{R_{1},\cdots,R_{n}\}$ adapted to $F$ and $\mu$.
This means, following Papadopoulos \cite{Pap91}, that each element $R_{j}$, $j=1,\cdots,n$, is a rectangle with two opposite sides contained in leaves of $F$ and two other opposite sides contained in leaves of $\mu$ and such that the union of these rectangles covers the complement of the non-foliated regions and such that all the rectangles have disjoint interiors.
Each rectangle in $\mathcal{R}$ is foliated by leaves of $F$ and is possibly crossed by leaves of $\mu$ transversely to this foliation.
It is explained in \cite{Pap91} how to produce such a rectangular covering.  
Note that the transverse measures of all compact arcs joining transversely to the leaves of $F$ the non-geodesic opposite sides of a given rectangle are the same.
Now each rectangle of $\mathcal{R}$ is crossed by horogeodesic segments of $\alpha^{\star}$.
Define the length of $\alpha^{\star}$ by summing up the lengths of the horogeodesic segments in each rectangles, namely,
$$
\ell_{h}(\alpha^{\star})=\sum_{i=1}^{n}\int_{\textrm{segments\ of}\ R_{i}\cap\alpha^{\star}}\ell_{h}(\alpha^{\star}(x))d\alpha^{\star}(x),
$$
where $\alpha^{\star}(x)$ is a horocyclic segment of $\alpha^{\star}\cap R_{i}$ and $d\alpha^{\star}$ is the transverse measure on $\alpha^{\star}$ induced by $\alpha$.

We have to show that $\ell_{h}(\alpha^{\star})$ is well-defined.

\begin{lemma}
\label{lemma:length_ineq}
Let $\alpha^{\star}$ be a horogeodesic lamination in the measure class of the measured geodesic lamination $\alpha$.
The length of the horogeodesic lamination $\alpha^{\star}$ does not depend upon the choice of a rectangular covering $\mathcal{R}$ of $\Sigma$.
Moreover, we have 
$$\ell_{h}(\alpha^{\star})\geq\ell_{h}(\alpha).$$
\end{lemma}

\begin{proof}
Let $\mathcal{R}$ and $\mathcal{R}'$ be two rectangular coverings of $\Sigma$.
The rectangles $R'_{1},\cdots,R'_{m}$ decompose each rectangle $R_{j}$ of $\mathcal{R}$ into smaller rectangles $R_{j,1},\cdots,R_{j,n_{j}}$ by looking at the intersections of the rectangles of $\mathcal{R}'$ with $R_{j}$.
It is clear that the length of each leaf $\alpha^{\star}(x)$ of $\alpha^{\star}$ crossing $R_{j}$ is the sum of the lengths of the leaves $\alpha^{\star}(x)\cap R'_{j,k}$, $k=1,\cdots,n_{j}$.
In symbols, 
$$
\ell_{h}(\alpha^{\star}(x))=\sum_{k=1}^{n_{j}}\ell_{h}(\alpha^{\star}(x)\cap R'_{j,k}).
$$
Let us denote by $\ell_{h}^{\mathcal{R}}(\alpha^{\star})$ the length of $\alpha^{\star}$ computed with the rectangular covering $\mathcal{R}$.
The family of small rectangles $R'_{j,k}$, $1\leq j\leq n$, $1\leq k\leq n_{j}$, is a rectangular covering $\mathcal{R}\cap\mathcal{R}'$ of $\Sigma$.
By what has just been pointed out, we have 
$$
\ell_{h}^{\mathcal{R}}(\alpha^{\star})=\ell_{h}^{\mathcal{R}\cap\mathcal{R}'}(\alpha^{\star}).
$$
By combining the rectangles of $\mathcal{R}\cap\mathcal{R}'$ differently, one recovers the rectangular 
covering $\mathcal{R}\cap\mathcal{R}'$.
By combining accordingly the terms in the sum giving $\ell_{h}^{\mathcal{R}\cap\mathcal{R}'}(\alpha^{\star})$, this yields $\ell_{h}^{\mathcal{R}}(\alpha^{\star})=\ell_{h}^{\mathcal{R}'}(\alpha^{\star})$.
The length of $\alpha^{\star}$ thus does not depend upon the chosen rectangular covering of $\Sigma$.

The fact that the length of the horocyclic lamination is greater than the length of the geodesic lamination is a special case of  Proposition \ref{proposition:min_length}.
\end{proof}

The length $\ell_{h}(\alpha^{\star})$ splits into two components, namely, the length of the geodesic part and the length of the horocyclic part of $\alpha^{\star}$.
We write $I(\alpha^{\star},F)$ for the length of the geodesic part of $\alpha^{\star}$ and we call it the \textbf{intersection number} between $\alpha^{\star}$ and $F$ (Recall that we are using the shortcut $F=F_{\mu}(h)$, so the intersection number really depends on $h$).
We write $L_{h}(\alpha^{\star})$ for the length of the horocyclic part of $\alpha^{\star}$.
Thus
$$
\ell_{h}(\alpha^{\star})=I(\alpha^{\star},F)+L_{h}(\alpha^{\star}).
$$
Let $i(\cdot,\cdot)$ denote the geometric intersection number, defined over $\mathcal{ML}(\Sigma)\times\mathcal{ML}(\Sigma)$.
One can find its definition in \cite{Bonahon}, \cite{FLP79}, \cite{Thurston80}.
This intersection number also makes sense over $\mathcal{ML}(\Sigma)\times\mathcal{MF}(\Sigma)$ by the use of the canonical isomorphism between $\mathcal{ML}(\Sigma)$ and $\mathcal{MF}(\Sigma)$.

\begin{lemma}
\label{lemma:Intersection_numbers}
Let $\alpha^{\star}$ be a horogeodesic lamination in the measure class of the measured geodesic lamination $\alpha$.
We have 
$$
I(\alpha^{\star},F)\geq i(\alpha,F)
$$
and the equality $I(\alpha^{\star},F)=i(\alpha,F)$ holds if and only if $\alpha^{\star}$ is a good horogeodesic lamination.
\end{lemma}

\begin{proof}
By definition, we have $I(\alpha',F)\geq i(\alpha,F)$ for all laminations $\alpha'$ in the measure class of $\alpha$.
The existence of a disc bounded by leaves of $F$ and $\alpha^{\star}$ obviously implies that $I(\alpha^{\star},F)>i(\alpha,F)$, since the number $I(\alpha^{\star},F)$ can certainly be decreased by erasing the disc.
This shows that if the above equality holds, then the horogeodesic lamination $\alpha^{\star}$ is good.

It is easy to see that the infimum in the definition of $i(\alpha,F)$ can be taken over all horogeodesic laminations.
Recalling that erasing discs strictly decreases the intersection number, we can conclude that the infimum can be taken over all \emph{good} horogeodesic laminations.

Let $\alpha'$ and $\alpha''$ be two good horogeodesic laminations in the measure class of $\alpha$.
From Lemma \ref{lemma:moves}, one can pass from $\alpha'$ to $\alpha''$ by using pushes and these moves preserve the intersection number with the horocyclic foliation $F$.
We therefore have
$$
I(\alpha',F)=I(\alpha'',F),
$$
hence $I(\alpha',F)=i(\alpha,F)$ for all good measured horogeodesic laminations $\alpha'$ in the measure class of $\alpha$.
\end{proof}

\section{Some useful inequalities}

Let $\mu$ be a complete geodesic lamination of the surface $\Sigma$ and let $h$ be a hyperbolic metric on $\Sigma$.
Let $F_{\mu}(h)$ be the horocyclic foliation associated to $\mu$ and $h$ and let $\lambda_{\mu}(h)$ denote the corresponding geodesic lamination.
Consider a measured geodesic lamination $\alpha$ and construct a special horogeodesic lamination $\alpha^{\star}_{sp}$ in the measure class of $\alpha$ as follows.

If $\alpha$ is contained in $\mu$ or in $\lambda_{\mu}(h)$, then we set $\alpha^{\star}_{sp}=\alpha$.

Suppose now that $\alpha$ is transverse to $\mu$.
Fix a rectangular covering $\mathcal{R}=\{R_{1},\cdots,R_{n}\}$ adapted to $\mu$ and $F_{\mu}(h)$.
A component of the intersection between the geodesic lamination $\alpha$ and a rectangle $R_{j}$ of $\mathcal{R}$ is a geodesic segment $c$ whose endpoints either 
\begin{itemize}
\item both lie in one non-geodesic side of $R_{j}$, or
\item connect the two opposite non-geodesic sides of $R_{j}$, or
\item connect the two opposite geodesic sides of $R_{j}$, or
\item connect one non-geodesic side of $R_{j}$ with one geodesic side of $R_{j}$.
\end{itemize}

If one of the last three cases occurs, we replace the geodesic segment $c$ by a \emph{good} horogeodesic curve $c^{\star}$ with the same endpoints, made up of one horocyclic piece and one geodesic piece;
the curve $c^{\star}$ is contained in $R_j$.
(Note that the third case can be reduced to the two first cases by subdiving $R_{j}$ along a horocyclic leaf of $F$.)
If the first case occurs, we replace the geodesic segment $c$ by a horogeodesic curve contained in $R_j$ with the same endpoints but which is \emph{not} good.
This horogeodesic curve $c^{\star}$ is a needle, made up of one piece contained in the non-geodesic boundary of the rectangle, followed by a geodesic segment perpendicular to the leaves of $F_{\mu}(h)$ joining the non-geodesic boundary to the farthest point of $c$ and finally followed by a piece contained in the non-geodesic boundary.
See Figure \ref{Figure7}.

\begin{figure}[!h]
\centering
\includegraphics[width = 10cm, height = 3cm]{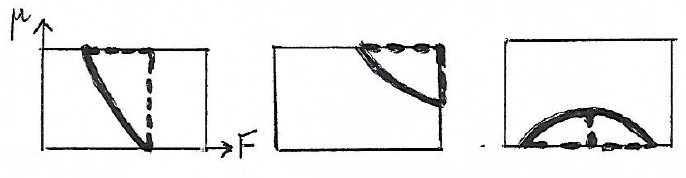}
\caption{Replacement of $c$.}
\label{Figure7}
\end{figure}

If $\alpha$ crosses a non-foliated region, we modify $\alpha$ so that it follows the boundary (and this way avoids the interior) of the non-foliated region.

The main observation to be made here lies in the following inequality.
$$
\ell_{h}(c^{\star})\geq\ell_{h}(c).
$$
This inequality comes from the fact (relying on a quick argument of hyperbolic geometry) that the length of a geodesic segment crossing a spike is greater or equal to the length of its projection along $F_\mu(h)$ onto the boundary of the spike.

Summing up the lengths of all geodesic arcs of $\alpha\cap R_{j}$ and summing up again over all rectangles of $\mathcal{R}$, we obtain the inequality
$$
I(\alpha^{\star}_{sp},\lambda_{\mu}(h)) \leq \ell_{h}(\alpha).
$$

Now, according to Lemma \ref{lemma:stairstep_good}, we can homotope $\alpha^{\star}_{esp}$ into a good horogeodesic curve $\alpha^{\star}$.
Lemma \ref{lemma:Intersection_numbers} implies that $i(\alpha,\lambda_{\mu}(h)) = I(\alpha^{\star},\lambda_{\mu}(h)) \leq I(\alpha^{\star}_{sp},\lambda_{\mu}(h))$.
We thus get the following inequalities.

\begin{proposition}
\label{proposition:inequalities}
Let $h$ be a hyperbolic metric on $\Sigma$ and $\mu$ be a complete geodesic lamination of $\Sigma$.
Let $\alpha$ be a measured geodesic lamination of the surface $\Sigma$ and let $\alpha^{\star}$
be a good horogeodesic lamination in the class of $\alpha$.
The following inequalities hold.
$$
i(\alpha,\lambda_{\mu}(h)) 
\leq \ell_{h}(\alpha) 
\leq 
i(\alpha,\lambda_{\mu}(h))+L_{h}(\alpha^{\star}).
$$
\end{proposition}

\begin{proof}
From the discussion above we draw the first inequality 
$$
i(\alpha,\lambda_{\mu}(h)) = I(\alpha^{\star},\lambda_{\mu}(h)) 
\leq \ell_{h}(\alpha).
$$
Lemma \ref{proposition:min_length} implies that 
\begin{eqnarray*}
\ell_{h}(\alpha) \leq \ell_{h}(\alpha^{\star}) & = & I(\alpha^{\star},\lambda_{\mu}(h))+L_{h}(\alpha^{\star})\\
& = & i(\alpha,\lambda_{\mu}(h)) + L_{h}(\alpha^{\star}).
\end{eqnarray*}
\end{proof}

\section{Geodesic laminations disjoint from the horocyclic lamination}

Suppose that the surface $\Sigma$ is equiped with a hyperbolic metric $h$ and with a complete geodesic lamination $\mu$. 
As before, $\lambda = \lambda_\mu(h)$ denotes the measured geodesic lamination associated to the horocyclic foliation $F = F_{\mu}(h)$.

Let $G_F$ be the singular graph of the partial foliation $F$, that is, the union of all pieces of leaves of $F$ connecting two distinguished points.
The graph $G_F$ is compact, at most trivalent.
It might be homotopic to a point, in which case we consider $G_F$ as empty.

\begin{lemma}
\label{lemma:PositiveIntersection}
Let $\alpha$ be a measured geodesic lamination and let $\alpha^{\star}$ be a good horogeodesic lamination in the measure class of $\alpha$.
Then $i(\alpha,\lambda)=0$ if and only if each leaf of $\alpha^{\star}$ is contained in $G_F$ or is a leaf of $F$.
In particular, if the lamination $\alpha$ is disjoint from $\lambda$, then $\alpha^{\star}$ is contained in $G_F$.
If the lamination $\alpha$ has a transverse intersection with $\lambda$, then $i(\alpha,\lambda)>0$.
\end{lemma}

\begin{proof}
It is clear that if every leaf of $\alpha^{\star}$ is either contained in $G_F$ or is a leaf of $F$, then $i(\alpha,\lambda)=I(\alpha^{\star},\lambda)=0$.

Conversely, suppose that there exists a leaf of $\alpha^{\star}$ which is not a leaf of $F$ nor contained in $G_F$.
This leaf then contains a geodesic piece.  
Up to pushes, we can assume that this leaf is a stairstep leaf $\gamma$ and that the geodesic piece is a shift segment $c$.
The length of this geodesic piece is its intersection number $i(c,F)$ with respect to the transverse measure of $F$.
The leaf $\gamma$ either corresponds to an isolated leaf or to a dense leaf of $\alpha$.
In the first case, the intersection number $i(c,F)$ contributes positively to $i(\alpha,\lambda) = I(\alpha^{\star},F)$.
In the second case, there will be a continuum of other leaves of $\alpha$ following $\gamma$ through the same spikes about $c$, and each of them yields the same shift segment when they are replaced by a stairstep curve in $\alpha^{\star}$.  
All these shift segments contribute positively to $i(\alpha,\lambda) = I(\alpha^{\star},F)$ as well.
We therefore have $i(\alpha,\lambda)>0$, which proves the converse.

Now if the lamination $\alpha$ is disjoint from $\lambda$, then $i(\alpha,\lambda)=0$ and every leaf of $\alpha^{\star}$ is contained in $G_F$ or is a leaf of $F$.
But disjointness implies that no leaf of $\alpha^{\star}$ can be a leaf of $F$.
The proof is over.
\end{proof}

\section{Proof of the main theorem}

Let $\mu$ be a complete geodesic lamination on $\Sigma$ and let $h$ be a hyperbolic metric on $\Sigma$.
Consider the stretch ray $t\mapsto h_{t}$ supported by $\mu$ and emanating from $h_{0}=h$ for $t=0$.
Set as before $F=F_{\mu}(h)$ and let $\lambda$ denote the projective class of the geodesic lamination canonically associated to $F$.
We shall now prove Theorem \ref{theorem:behavior} stated in the introduction.

\begin{theorem}
Let $t\mapsto h_{t}$, $t\in\mathbb{R}_{\geq0}$, be a stretch ray of Teichm\"uller space $\mathcal{T}(\Sigma)$ whose support is denoted by $\mu$ and its direction by $\lambda$.
Let $\alpha$ be a measured geodesic lamination of $\Sigma$.
\begin{itemize}
\item If $\alpha$ has non-empty transverse intersection with $\lambda$, then
$\displaystyle{\lim_{t\to+\infty}}\ell_{h_{t}}(\alpha)=+\infty$.
\item If $\alpha$ is contained in the direction $\lambda$, then $\displaystyle{\lim_{t\to+\infty}}\ell_{h_{t}}(\alpha)=0$.
\item If $\alpha$ has empty intersection with the direction $\lambda$, then the function $t\mapsto\ell_{h_{t}}(\alpha)$ is bounded in $]0 ; +\infty[$.
\end{itemize}
\end{theorem}

\begin{proof}
The case where $\alpha$ is contained in $\lambda$ has already been dealt with in Lemma \ref{lemma:lambda_mu}.
Suppose that $\alpha$ has non-empty transverse intersection with the direction $\lambda$.
By Proposition \ref{proposition:inequalities}, we have, for all $t\geq0$,
$$
i(\alpha,F_{\mu}(h_{t}))\leq \ell_{h_{t}}(\alpha).
$$
Since $\alpha$ is not contained in $\lambda$, Lemma \ref{lemma:PositiveIntersection} says that $i(\alpha,F_{\mu}(h))>0$.
Now by definition,
$$
i(\alpha,F_{\mu}(h_{t}))=e^{t}i(\alpha,F_{\mu}(h)).
$$
Therefore, $\displaystyle{\lim_{t\to+\infty}}i(\alpha,F_{\mu}(h_{t}))=+\infty$ and hence $\displaystyle{\lim_{t\to+\infty}}\ell_{h_{t}}(\alpha)=+\infty$.\\

Let us now deal with the last case, where $\alpha$ has empty intersection with $\lambda$.
We have $i(\alpha,\lambda)=0$.
Proposition \ref{proposition:inequalities} gives, for any $h$,
$$
0\leq \ell_{h}(\alpha) \leq L_{h}(\alpha^{\star}),
$$
where $\alpha^{\star}$ is a good horogeodesic lamination representing $\alpha$, which is contained in $G_F$.
Once $\alpha^{\star}$ has been chosen for some $t\geq0$, this choice is canonically carried over for all the other values of $t$.
It is then rather easy to see that $L_{h_t}(\alpha^{\star})$ is decreasing and bounded as $t$ increases (see Lemma \ref{lemma:asymptotic_length_horogeodesic}).
Hence $t\mapsto\ell_{h_{t}}(\alpha)$ is bounded from above.
It remains to show that the function $t\mapsto\ell_{h_{t}}(\alpha)$ is also bounded from below away from zero.
To do this, we go back to the inequality 
$$
I(F_{\mu}(h_{t}),\alpha^{\star}_{sp})\leq \ell_{h_{t}}(\alpha),
$$ 
for some horogeodesic lamination $\alpha^{\star}_{sp}$ in the class of $\alpha$ which is \emph{not} good (as described above).
The non-goodness of $\alpha^{\star}_{sp}$ implies that $I(F_{\mu}(h_{t}),\alpha^{\star}_{sp})>0$ for all $t\geq0$.
It remains to show that $I(F_{\mu}(h_{t}),\alpha^{\star}_{sp})$ is bounded below by a positive $t$-independent constant.
We reason by contradiction.
Suppose that $I(F_{\mu}(h_{t}),\alpha^{\star}_{sp})$ converges to zero for some sequence of values of $t$ increasing to infinity.
This implies that the lengths of almost all geodesic segments of $\alpha^{\star}_{sp}$ converge to zero.
Up to erasing all needles of $\alpha^{\star}_{sp}$ with pushes, we can assume that all geodesic segments of $\alpha^{\star}_{sp}$ converge to zero.
Hence, for any $\epsilon>0$ and for any $t$ great enough, each leaf of $\alpha^{\star}_{sp}$ is contained in an $\epsilon$-neighbourhood of a well-defined singular leaf of $F$.
For $t$ great enough, this singular leaf is very close to a concatenation of boundaries of non-foliated regions.
But such a curve cannot be close to any geodesic, which is a contradiction.
Hence there exists a lower bound (which I think does not depend upon $\alpha$) to $I(F_{\mu}(h_{t}),\alpha^{\star}_{sp})$.
This concludes the proof.
\end{proof}

We used the following result.

\begin{lemma}
\label{lemma:asymptotic_length_horogeodesic}
Let $\alpha$ be a measured geodesic lamination and let $\alpha^{\star}$ be a good horogeodesic lamination in the same measure class as $\alpha$.
Let $t\mapsto h_{t}$, $t\geq0$, be a stretch ray of Teichm\"uller space $\mathcal{T}(\Sigma)$ whose support is denoted by $\mu$ and its direction by $\lambda$.
The length function $t\mapsto L_{h_{t}}(\alpha^{\star})$ is decreases as $t$ increases and is therefore bounded from above.
\end{lemma}

\begin{proof}
By definition of stretches (see \cite{Thurston86}), for given $t\leq t'$, each piece of horocycle of $\alpha^{\star}$ for the metric $h_t$ which is $d$-away from a non-foliated region of $F$ is carried over for the metric $h_{t'}$ to the piece of horocycle that lies $e^{t'-t}d$ from the same non-foliated region. 
This implies that the length of every horocyclic piece of $\alpha^{\star}$ decreases as $t$ increases. 
Hence, $t\mapsto L_{h_{t}}(\alpha^{\star})$ is decreases as $t$ increases.
\end{proof}

\begin{remark}
It can be shown that if the singular leaf $f$ of $F$ contained in $G_F$ is a concatenation of sides of non-foliated regions alternately lying on one side of $f$ then on the other, then the geodesic leaf of $\alpha$ homotopic to $f$ passes through all distinguished points of $f$.
In this case, $I(F_{\mu}(h_{t}),\alpha^{\star}_{sp})$ and $\ell_{h_{t}}(\alpha)$ can be explicitely calculated.
%The best universal constant can be computed and is equal to $\frac{1}{2}\ln(5)$.
\end{remark}

\end{document}